\documentclass[12pt,english]{article}
\usepackage{amsmath,amssymb,amsthm}
\usepackage{setspace}
\usepackage{babel}

\theoremstyle{plain}
\newtheorem{theorem}{Theorem}
\newtheorem*{theorem*}{Theorem}
\newtheorem{lemma}{Lemma}

\newtheorem*{corollary*}{Corollary}
\newtheorem{definition}{Definition}
\newtheorem*{remark*}{Remark}

\def\mathscr{\mathfrak}

\newcommand{\ds}{\displaystyle}
\newcommand{\ov}{\overline}
\def\R{{\mathbb R}}
\def\N{{\mathbb N}}
\def\Z{{\mathbb Z}}
\def\C{{\mathbb C}}

\def\s{\sigma}
\def\l{\lambda}

\def\E{\mathcal{E}}
\def\D{\mathcal{D}}
\def\L{\mathcal{L}}
\def\LL{\Lambda}
\def\A{\mathbb{A}}
\def\G{\mathbb{G}}
\def\P{\mathbb{P}}
\def\EE{\mathbb{E}}
\def\d{\partial}
\def\cxi{\widetilde{\xi}}
\def\S{\widetilde{S}}
\def\SS{\mathcal{S}}
\def\NN{\mathcal{N}}

\date{}

\begin{document}

\title{The asymptotic trace norm of random circulants and the graph energy}

\author{Sergiy Koshkin\\
 Department of Mathematics and Statistics\\
 University of Houston-Downtown\\
 One Main Street\\
 Houston, TX 77002\\
 e-mail: koshkins@uhd.edu}
\maketitle
\begin{abstract}\

We compute the expected normalized trace norm (matrix/graph energy) of random symmetric band circulant matrices and graphs in the limit of large sizes, and obtain explicit bounds on the rate of convergence to the limit, and on the probabilities of large deviations. We also show that random symmetric band T\"oeplitz matrices have the same limit norm assuming that their band widths remain small relative to their sizes. We compare the limit norms across a range of related random matrix and graph ensembles. 
\bigskip

\end{abstract}

\section{Introduction}\label{S0}

The energy of a graph was originally introduced by Gutman, and was motivated by applications to organic chemistry, see 
\cite[1.1]{LSG} and references therein. It is a spectral invariant of a graph equal to the trace norm of its adjacency matrix, i.e. to the sum of its singular values. Gutman's original conjecture that complete graphs have the greatest energy among all graphs with the same number of vertices $n$ was disproved dramatically by Nikiforov \cite{Nik}, see also \cite{DLL}, who showed that for large $n$ the energy of almost all graphs is greater. Nikiforov's result was based on identifying the ensemble of random graphs on $n$ vertices, or rather their adjacency matrices, with a particular instance of the Wigner ensemble \cite{Bai}. The normalized asymptotic graph energy was then computed using its limit spectral distribution, the semicircle law. After Nikiforov's paper the term "matrix energy" is sometimes used to call the trace norm of a matrix, even if it is not an adjacency matrix of a graph. For Hermitian matrices, up to normalization, the matrix energy is none other than the first absolute moment of the matrix's spectral distribution.

The limit spectral distributions are known for a number of other random matrix ensembles, some of them corresponding to  interesting ensembles of random graphs, particularly ensembles of band symmetric \cite{MPK}, symmetric circulant 
\cite{BM}, symmetric T\"oeplitz \cite{HM}, and symmetric band T\"oeplitz matrices \cite{Karg}. The limit spectral distribution for the ensembles of regular random graphs are also known \cite{McK, TVW}. In many cases they are found by the Wigner's original method of moments \cite{Bai,DLL}, implying that the distributions converge to the limit with all their absolute moments. Therefore, once the limit distribution is known the asymptotic energy can be found by elementary integration analogous to Nikiforov's. However, symmetric circulants and band T\"oeplitz random matrices have so few independent entries that the method of moments does not work for them. But the eigenvalues of circulants can be expressed explicitly as weighted sums of their entries, so the Central Limit Theorem (CLT) can be used instead. For symmetric circulants convergence in distribution to a Gaussian limit was established in \cite{BM}, see also \cite{Meck} for the non--Hermitian case. For symmetric T\"oeplitz matrices without the band structure the method of moments does work, the limit spectral distribution is non--Gaussian with sub--Gaussian even moments \cite{HM}. 

For symmetric band T\"oeplitz matrices and symmetric band circulants convergence in distribution to the same Gaussian limit is proved in \cite{Karg}, assuming their band widths remain small relative to their sizes. The proof relies on modifying band T\"oeplitz matrices into band circulants in an asymptotically negligible way using a classical trick \cite[4.3]{Gray}, and then applying CLT. However, an extra limit has to be taken after CLT, and convergence of moments (even of variances) does not follow, let alone any estimates on its rate. Moreover, it is assumed in \cite{BM}, \cite{Karg} that the random entries are centered, whereas for graph adjacency matrices they are Bernoullian with positive means. Surprisingly, it is dealing with the non-zero means that turns out to be the hardest because of the singular limit behavior of the Dirichlet kernel, see Section \ref{S4}. 

In this paper we will establish convergence of the trace norms (matrix energies) to the first absolute moments of the limit spectral distribution for band symmetric circulant (and T\"oeplitz) random matrices and graphs. Moreover, we will produce explicit estimates on the rate of convergence and on the probabilities of large deviations. We chose to focus the exposition on circulants rather than T\"oeplitz matrices not only because the estimates are cleaner, but also because circulant graphs are easier to visualize, and their energies were a subject of much research lately \cite{BSh,IlBa,LeSa,SaSa}.

A graph is called a circulant if its vertices can be identified with vertices of a regular polygon, and its edges, with some of its sides and diagonals, in such a way that geometric symmetries of the polygon induce graph isomorphisms. In particular, all vertices look the same as far as connections to other vertices are concerned. The band condition means that the edges can only join vertices located at a bounded distance from each other along the perimeter of the polygon. T\"oeplitz graphs are trickier to describe, see \cite{D-Z}, roughly they are obtained from circulants by removing a few vertices and all adjacent edges so that the polygon opens up. Most of the recent studies of the graph energy of circulants use number theoretic methods, our work complements them with analytic and probabilistic approaches. In particular, we rely on analytic properties of the Dirichlet kernel, the Berry-Esseen estimates and a Talagrand concentration inequality.

The paper is organized as follows. In Section \ref{S1} we introduce the notation and terminology and give precise formulations of our main results for matrix ensembles, Theorems \ref{AnbThm}--\ref{ToepCirc}. In Section \ref{S2} the matrix results are applied to circulant graphs, and we compare their expected asymptotic energy to the energies of other graph ensembles, general, band, band T\"oeplitz, regular, and to the average energy of circulants with a fixed degree. In Section \ref{S3} we outline the main steps in the proofs of Theorems \ref{AnbThm} and \ref{LrgDevThm} with technical details worked out in Sections \ref{S4}--\ref{S6}, Theorem \ref{ToepCirc} is proved in Section \ref{S7}. In Section \ref{S8} we discuss rate estimates, convergence of higher moments and a generalization to block-circulant matrices and graphs.

\section{Preliminaries and main results}\label{S1}

In this section we introduce the necessary notation and terminology from linear algebra, probability theory and graph theory that are used in the paper. We also give precise statements of our main results, Theorems \ref{AnbThm}--\ref{ToepCirc}.
\begin{definition}\label{circmatrix} An $n\times n$ matrix $A=(a_{ij})$ is called a {\sl circulant} if each of its rows is the right cyclic shift of the row above, i.e. $a_{ij}=a_{(j-i)\!\!\mod n}$ for some tuple of $a_k\in\C$, $k=0,\dots,n-1$. We denote 
$$
\text{\rm Circ}(a_0,\dots,a_{n-1}):=\big(a_{(j-i)\!\!\mod n}\big)
=\begin{pmatrix}a_0&a_1&a_2&\ldots&a_{n-1}\\
a_{n-1}&a_0&a_1&\ldots&a_{n-2}\\
a_{n-2}&a_{n-1}&a_0&\ldots&a_{n-3}\\
\vdots&\vdots&\vdots&\ddots&\vdots\\
a_1&a_2&a_3&\ldots&a_0\\
\end{pmatrix}.
$$
A circulant $A$ is {\sl band with band width $b$} if $a_k=0$ for $b<k<n-b$\,. 
\end{definition}
\noindent Let $A:=\text{\rm Circ}(a_0,\dots,a_{n-1})$, it is well-known \cite[3.1]{Gray}, \cite[8.6]{Lan} that the eigenvalues of $A$ are $\l_r(A)=\sum_{k=0}^{n-1}a_k\omega^{k\,r}$, where $\omega$ is any primitive $n$-th root of unity, e.g. $\omega=e^{\frac{2\pi i}n}$. $A$ is Hermitian if and only if $a_{n-k}=\ov{a}_k$ for $k:=1,\dots,n-1$ and $a_0\in\R$ (for even $n$ also $a_{\frac{n}2}\in\R$). For Hermitian $A$ we can write the eigenvalues in explicitly real form
\begin{equation}\label{HermEigen}
\l_r(A)=\begin{cases}
\ds{a_0+(-1)^ra_{\frac{n}2}+2\,\sum_{k=1}^{\frac{n}2-1}\,\textrm{Re}\big(a_k\,e^{\frac{2\pi i}n\,k\,r}\big)}, &n\text{ even}\\
\ds{a_0+2\,\sum_{k=1}^{\frac{n-1}2}\,\textrm{Re}\big(a_k\,e^{\frac{2\pi i}n\,k\,r}\big)}, &n\text{ odd}.
\end{cases}
\end{equation}
From now on we assume that $A$ is a real symmetric circulant, i.e. the tuple $a_k$ is palindromic and all $a_k\in\R$, and for simplicity we also assume $a_0=0$, and $a_{\frac{n}2}=0$ for even $n$. Then \eqref{HermEigen} simplifies to 
\begin{equation}\label{SymEigen}
\l_r(A)=2\,\sum_{k=1}^{\lfloor\frac{n-1}2\rfloor}a_k\cos\Big(2\pi k\,\frac{r}n\Big)\,
\end{equation}
for all $n\in\N$\,, where $\lfloor x\rfloor$ returns the largest integer smaller than or equal to $x$. Under our assumptions for real symmetric band circulants with band width $b$ to exist one must have 
$b<\frac{n}2$, and for such circulants $b$ can replace $\lfloor\frac{n-1}2\rfloor$ in the upper summation limit of \eqref{SymEigen}\,.

In the next definition we follow the terminology of \cite{Nik}, which is common in the literature on graph energy, but what is defined is better known in linear algebra and functional analysis as the {\sl trace norm} of 
a matrix \cite[III.7]{GK}.  
\begin{definition} The {\sl matrix energy} of an $n\times n$ matrix $A$ is $\E(A):={\rm tr}(|A|)=\sum_{r=1}^{n}s_r(A)$, where $s_r(A)$ are the singular values of $A$, i.e. the positive square roots of the eigenvalues of $A^*A$, and $|A|:=(A^*A)^{\frac12}$.
\end{definition}
\noindent When $A$ is Hermitian, in particular real symmetric as in our case, $s_r(A)=|\l_r(A)|$ \cite[II.2]{GK}, so 
$\E(A):=\sum_{r=1}^{n}|\l_r(A)|$, where $\l_r(A)$ are the eigenvalues of $A$. For real symmetric band circulants $A=\text{\rm Circ}(a_0,\dots,a_{n-1})$ with $b<\frac{n}2$ and $a_0=0$ one obtains from \eqref{SymEigen}
\begin{equation}\label{EnrgCirc}
\E(A)=2\,\sum_{r=1}^n\left|\sum_{k=1}^ba_k\cos\Big(2\pi k\,\frac{r}n\Big)\right|.
\end{equation}
We will be interested in the asymptotic behavior of the matrix energy when $a_k$ are random variables and $n,b\to\infty$. Our assumptions that $a_0=0$ and $a_{\frac{n}2}=0$ will not affect the results since their contributions are asymptotically negligible.
\begin{definition} Denote by $\A_{n,b}$ the {\sl ensemble of $n\times n$ random real symmetric band circulants} with band width $b<\frac{n}2$ and $a_0=0$ given by $\text{\rm Circ}(0,\xi_1,\dots,\xi_b,0,\dots,0,\xi_b,\dots,\xi_1)$, where $\xi_k$ are independent identically distributed real random variables with expected value $a:=\EE\xi_k$ and finite variance $\s^2:=\EE|\xi_k-a|^2$\,. By abuse of notation $\A_{n,b}$ will also denote a random element of the ensemble.
\end{definition}
\noindent From \eqref{EnrgCirc} the normalized matrix energy of a random circulant is
\begin{equation}\label{NormEnrgCirc}
\frac1n\,\E(\A_{n,b})=2\,\sum_{r=1}^n\left|\sum_{k=1}^ba_k\cos\Big(2\pi k\,\frac{r}n\Big)\right|\frac1n\,.
\end{equation}
\noindent Denoting $S_b(t):=\sum_{k=1}^ba_k\cos\big(2\pi k\,t\big)$ we see that the right hand side of \eqref{NormEnrgCirc} is twice a Riemann sum of $|S_b(t)|$ on $[0,1]$. Since $|S_b(t)|$ is obviously continuous its Riemann sums converge to the Riemann integral $\frac1n\,\E(\A_{n,b})\xrightarrow[n\to\infty]{}2\int_0^1|S_b(t)|\,dt$. The idea of treating normalized spectral functions of T\"oplitz matrices and circulants as Riemann sums, and computing their limit values as Riemann integrals, is classical and goes back to Szeg\"o \cite[4.3]{GSz}. In its turn, $S_b(t)$ for each $t\in[0,1]$ is a sum of independent real random variables with finite variances, albeit not identically distributed. One expects from the Centeral Limit Theorem (CLT) \cite[III.4.1]{Shir} that $\frac1{\sqrt{b}}\,S_b(t)$ is asymptotically Gaussian when $b\to\infty$. 

Thus, it is reasonable to conjecture that the expected values $\frac1{n\sqrt{b}}\,\EE\,\E(\A_{n,b})$ have a limit when $n,b\to\infty$, and that the sample values $\frac1{n\sqrt{b}}\,\E(\A_{n,b})$ concentrate around the limit values with high probability for large $n,b$. The main purpose of this paper is to establish these facts with explicit estimates on the rates of convergence and on the probabilities of large deviations. Two technical issues are immediately apparent. First, we do not assume that the random variables are centered, so the means will have to be controlled separately. Second, our heuristic argument above applied to the repeated limit $n\to\infty$ and then $b\to\infty$, not to the double limit $n,b\to\infty$. To prove that the double limit exists one needs a uniform in $n$ bound on the difference between the Riemann sums and the integrals as 
$b\to\infty$. It is not obvious that $\frac1{\sqrt{b}}\,S_b(t)$ admit such a bound, and therefore that the double limit exists at all.

The limit spectral distributions for ensembles of large random circulant and related matrices were considered by several authors recently \cite{BM,BD,HM,Meck}, and the limit spectral distribution for real symmetric band circulant matrices is proved to be Gaussian in \cite{Karg}. The last work does not use the method of moments, common for the Wigner and similar matrix ensembles \cite[2.1.1]{Bai}, and only establishes convergence in distribution for centered random variables 
$\xi_k$. We on the other hand are interested in convergence of the first absolute moments (matrix energies) for off-centered variables, and explicit estimates on its rate and probabilities of large deviations. The approach of 
\cite{Karg} is therefore unsuitable for our purposes. We now state our main asymptotic result.
\begin{theorem}\label{AnbThm} Let the ensemble $\A_{n,b}$ be specified by independent identically distributed real random variables $\xi_k$ with the mean $\EE\xi_k=:a$, the variance $\EE|\xi_k-a|^2=:\s^2<\infty$, and the central third moment 
$\EE|\xi_k-a|^3=:\mu_3<\infty$\,. Then
\begin{multline}\label{AnbEst}
\left|\frac1{n\sqrt{b}}\,\EE\E(\A_{n,b})-\frac{\,2\,\s\,}{\sqrt{\pi}}\right|
\leq\frac{8\pi C_1}{3\sqrt{3}}\,\frac{\mu_3}{\s^2\sqrt{b}}
+\frac{\,4\,}{\sqrt{b}}\left(|a|+\frac\s{\sqrt{\pi b}}\right)\left(1+\frac1{\pi^2}\ln b\right)\\
+\frac{\,2\,}{n\sqrt{b}}\left(|a|+\frac{2\,\s}{\sqrt{\pi b}}\right)\left(1+(2b+1)(1+\gamma+\ln b)\right)
= O\left(\frac{\ln b}{\sqrt{b}}\right)\,,
\end{multline}
where $C_1<31.954$ and $\gamma$ is the Euler constant.
\end{theorem}
\noindent The rate of convergence can be improved to $O\left(\frac1{\sqrt{b}}\right)$ if $\frac{\,2\,\s\,}{\sqrt{\pi}}$ is replaced by a $b$ and $n$ dependent expression obtained by computing the first absolute moments of the Gaussian variables off--centered by the means of $\frac1{\sqrt{b}}\sum_{k=1}^b\xi_k\cos\left(2\pi k\,\frac{r}n\right)$. We do not pursue such improvement here.

To estimate the probabilities of large deviations we follow the approach of \cite{GZ} based on a concentration inequality of Talagrand \cite{Tal}. It requires stronger assumptions on the random variables than finite third moments, but provides  exponential bounds in return. In view of applications to graph theory we will assume that $\xi_k$ take values in a finite interval. There is also a version of Talagrand inequalities for variables with distributions satisfying a log--Sobolev inequality \cite{GZ}. For variables only assumed to have finite moments polynomial bounds can be derived as e.g. in \cite[6.1]{HM}. 
\begin{theorem}\label{LrgDevThm} Let the ensemble $\A_{n,b}$ be specified by independent identically distributed random variables $\xi_k$ taking values in the interval $[0,R]$, then 
\begin{equation}\label{TalaDevEst}
\mathbb{P}\left\{\left|\frac1{n\sqrt{b}}\,\E(\A_{n,b})-\frac1{n\sqrt{b}}\,\EE\E(\A_{n,b})\right|\geq\delta\right\}\leq4\,e^{-\frac{b}{8R^2}\big(\delta-\delta_0(b)\big)^2}\,, 
\end{equation}
where $\delta_0(b):=4\sqrt{\frac{2\pi R^2}b}$\,.
\end{theorem}
\noindent From Theorems \ref{AnbThm}, \ref{LrgDevThm} one can conclude that almost surely
\begin{equation}\label{AnbAsEnrg}
\E(\A_{n,b})=n\sqrt{b}\left(\frac{\,2\,\s\,}{\sqrt{\pi}}+o(1)\right),
\end{equation}
a weaker but simpler statement of our main results.

As in \cite{Karg} we also consider a closely related class of T\"oeplitz matrices.
\begin{definition}\label{Toepmatrix} A matrix $A=(a_{ij})$ is called {\sl T\"oeplitz} if each of its main diagonals contains identical entries, in other words $a_{ij}=a_{j-i}$ for some tuple of $a_k\in\C$, $k=-(n-1),\dots,n-1$. 
We denote 
$$
\text{\rm T\"oep}(a_{-(n-1)},\dots,a_{-1},a_0,a_1,\dots,a_{n-1}):=\big(a_{j-i}\big)
=\begin{pmatrix}a_0&a_1&a_2&\ldots&a_{n-1}\\
a_{-1}&a_0&a_1&\ldots&a_{n-2}\\
a_{-2}&a_{-1}&a_0&\ldots&a_{n-3}\\
\vdots&\vdots&\vdots&\ddots&\vdots\\
a_{-(n-1)}&a_{-(n-2)}&a_{-(n-3)}&\ldots&a_0\\
\end{pmatrix}.
$$
For Hermitian (real symmetric) matrices $a_{-k}=\ov{a_k}$, and we shorten the notation to\\ $\text{\rm T\"oep}_s(a_0,a_1,\dots,a_{n-1})$\,. A T\"oeplitz matrix is {\sl band with band width $b$} if $a_k=0$ for $|k|>b$\,. 
\end{definition}
\noindent One can see from Definitions \ref{circmatrix} and \ref{Toepmatrix} that every circulant matrix is T\"oeplitz, but the converse is not true. However, if a T\"oeplitz matrix is band with band width $b<\frac{n}2$ then it can be turned into a circulant by altering entries in its upper right and lower left corners, see \cite[4.3]{Gray} and Section \ref{S7}.  
\begin{definition} Denote by $\widetilde{\A}_{n,b}$ the {\sl ensemble of $n\times n$ random real symmetric band T\"oeplitz matrices with band width $b<\frac{n}2$} and $a_0=0$ given by 
$\text{\rm T\"oep}_s(0,\xi_1,\dots,\xi_b,0,\dots,0)$, where $\xi_k$ are independent identically distributed real random variables with expected value $a:=\EE\xi_k$ and finite variance 
$\s^2:=\EE|\xi_k-a|^2$\,.
\end{definition}
\noindent The corner trick change is asymptotically negligible for matrix sizes large relative to band widths, so the limit spectral distributions are the same as long as $\frac{b}{n}\to0$. The same holds in the Hermitian (real symmetric) case \cite{Karg}. Moreover, we prove the following.
\begin{theorem}\label{ToepCirc} In conditions of Theorem \ref{AnbThm}
\begin{equation}\label{ToepCircEq}
\frac1{n\sqrt{b}}\,\left|\EE\E(\widetilde{\A}_{n,b})-\EE\E(\A_{n,b})\right|\leq\frac{(\EE\xi_k^2)^\frac12}n\sqrt{2\,b(b+1)}
\leq2\,(|a|+\s)\,\frac{b}n\,.
\end{equation}
\end{theorem}
\noindent Analogs of Theorems \ref{AnbThm}, \ref{LrgDevThm} for $\widetilde{\A}_{n,b}$ can now be readily stated, we leave formulations to the reader. The main difference is that in addition to $b\to\infty$ one also needs $\frac{b}n\to0$ for the asymptotics to hold.

\section{Comparison of random graph energies}\label{S2}

In this section we interpret the results of Theorems \ref{AnbThm}, \ref{LrgDevThm} for random circulant graphs, and compare them to the graph energies of other random graph ensembles. We start with a precise definition of the circulant graphs \cite{BSh}.
\begin{definition}\label{circgraph} Let $\Z_n$ denote the set of residue classes modulo $n$, 
$\Z_n^*:=\Z_n\backslash\{0\}$ and $J\subseteq\Z_n^*$ be a subset. A {\sl circulant graph generated by $J$} is the graph with vertices labeled by the elements of $\Z_n$ with the $i$--th and the $j$--th vertex joined by an edge if and only if $j-i\in J\cup-J$ (we do not consider directed edges so $J$ has to be symmetrized). Without loss of generality, one can choose a minimal $J=\{j_1,\dots,j_m\}$, where $j_1<\dots<j_m\leq n-j_m<\dots<n-j_1$. We call these $j_i$ the {\sl jump sizes}, and denote $G\langle J\rangle=G\langle\,j_1,\dots,j_m\rangle$ the circulant graph generated by $J$.
\end{definition}
\noindent If one thinks of the vertices of a circulant as vertices of a regular polygon with unit sides the jump sizes are the distances traveled from a vertex to other vertices joined with it by an edge, when moving counterclockwise along the perimeter. Since the jump sizes are the same for all vertices the isometries of the polygon induce graph isomorphisms of the circulant. The band width restriction means that edges can only join vertices within the distance $b$ of each other along the perimeter. 

Recall that the adjacency matrix of a graph $G$ with the vertices labeled $1,\dots,n$ is the $n\times n$ matrix $A_G$ with $a_{ij}=1$ if the $i$--th and the $j$--th vertices are joined by an edge, and $a_{ij}=0$ otherwise \cite[1.1]{LSG}.  Identifying the elements of $\Z_n$ with $1,\dots,n$ it is easy to see from Definitions \ref{circmatrix} and \ref{circgraph} that $G$ is a circulant graph if and only if $A_G$ is a symmetric circulant $0,1$ matrix with zeros on the diagonal. Specifically, $\ds{A_{G\langle\,j_1,\dots,j_m\rangle}=\textrm{Circ}(a_0,\dots,a_{n-1})}$ with $a_j=1$ provided $j=j_l$ or $j=n-j_l$ for some $l$, and $a_j=0$ otherwise.
\begin{definition} Denote by $\G_{n,b}(p)$ the {\sl ensemble of random circulant graphs on $n$ vertices with jump sizes bounded by $b<\frac{n}2$} defined as follows. For each number from $1$ to $b$ one independently chooses it as a jump size with probability $p$\,. Once  $j_1<\dots<j_m$ are so chosen the corresponding random circulant graph is $G\langle\,j_1,\dots,j_m\rangle$\,.
\end{definition}
\noindent The corresponding ensemble of the adjacency matrices $A_{\G_{n,b}(p)}$ is easily identified with the ensemble 
$\A_{n,b}$, where the random variables $\xi_k$ are Bernoullian with probability of success $p$ \cite[I.4.1]{Shir}. We will denote the latter $\A_{n,b}(p)$ to distinguish from the general case. 

Historically, the graph energy defined next was introduced before the matrix energy for general matrices, see references in \cite[1.1]{LSG}. 
\begin{definition} The {\sl graph energy} of a graph $G$ is defined to be the trace norm (matrix energy) of its adjacency matrix 
$\E(G):=\E(A_G)={\rm tr}(|A_G|)$.
\end{definition}
\noindent We consider the graph energy of large random circulant graphs. For $\A_{n,b}(p)$ the mean is $a:=\EE\xi_k=p$ and the variance is $\s^2:=\EE|\xi_k-a|^2=p\,(1-p)$. Therefore, formula \eqref{AnbAsEnrg} becomes
\begin{equation}\label{GnbpAsEnrg}
\E\Big(\G_{n,b}(p)\Big)=n\sqrt{b}\left(\frac{\,2\,}{\sqrt{\pi}}\sqrt{p\,(1-p)}+o(1)\right).
\end{equation}
\noindent Note that for the ensemble $\G_{n}(p)$ of all random graphs on $n$ vertices with each edge 
having probability $p$ the value is almost surely \cite{DLL}, \cite[6.1]{LSG}:
\begin{equation}\label{GnpAsEnrg}
\E\Big(\G_{n}(p)\Big)=n\sqrt{n}\left(\frac{\,8\,}{3\pi}\sqrt{p\,(1-p)}+o(1)\right),
\end{equation}
and $\frac{\,2\,}{\sqrt{\pi}}>\frac{\,8\,}{3\pi}$. However, since the largest permissible value of $b$ is of the order $\frac{n}2$ and $\frac{\,2\,}{\sqrt{2\pi}}<\frac{\,8\,}{3\pi}$ random circulant graphs are on average less energetic than the general ones. For the $p=\frac12$ case considered earlier by Nikiforov \cite{Nik} the asymptotic becomes
\begin{equation}\label{Gnp1/2AsEnrg}
\E\left(\G_{n}\left(\frac12\right)\right)=n\sqrt{n}\left(\frac{\,4\,}{3\pi}+o(1)\right).
\end{equation}
\noindent A graph $G$ is called hyperenergetic if $\E(G)>\E(K_n)=2n-2$, where $K_n$ is the complete graph on $n$ vertices \cite{Ghyp}, \cite[8.1]{LSG}. It follows from \eqref{GnbpAsEnrg} that for large $n,b$ almost all circulant graphs are hyperenergetic, just as almost all general graphs are for large $n$. 

When $p=\frac12$, i.e. every pair of vertices within distance $b$ of each other is as likely as not to be joined by an edge, we have 
\begin{equation}\label{Gnb1/2AsEnrg}
\E\left(\G_{n,b}\left(\frac12\right)\right)=n\sqrt{b}\left(\frac{\,1\,}{\sqrt{\pi}}+o(1)\right).
\end{equation}
This is similar to the average energy $\E_{av}$ of circulants defined in \cite{BSh}. However, for $\E_{av}$ the average is taken over the set $\mathcal{T}(n,d)$ of circulant graphs on $n$ vertices with exactly $d$ jump sizes, whereas we average over circulants with the values of jump sizes bounded by $b$. Still, the comparison is instructive since for $p=\frac12$  approximately half of $a_1,\dots,a_b$ are likely to be $1$-s, in which case half of $a_{n-b},\dots,a_{n-1}$ are also 
$1$-s, so $d\sim b$ in $\G_{n,b}\left(\frac12\right)$ with high probability. It is shown in \cite[Thm.5]{BSh} that 
\begin{equation}\label{TndAsEnrg}
n\sqrt{d}\left(\frac{\,1\,}{\sqrt{3}}+o(1)\right)\leq\E_{av}\Big(\mathcal{T}(n,d)\Big)
\leq n\sqrt{d}\Big(1+o(1)\Big)\,.
\end{equation}
Fixing the number of jump sizes allows equally probable edges to join distant vertices, whereas our ensemble always selects edges form a $b$-neighborhood of a vertex. Therefore, formulas \eqref{Gnb1/2AsEnrg}, \eqref{TndAsEnrg} and $\frac1{\sqrt{\pi}}<\frac1{\sqrt{3}}$ imply that more uniform distribution of edges increases the graph energy for large $n,d$.

By Theorem \ref{ToepCirc}, for band symmetric T\"oeplitz matrices formula \eqref{AnbAsEnrg} also holds if additionally
$\frac{b}n\xrightarrow[n,b\to\infty]{}0$. Let us compare their energy to that of general random band symmetric 
$n\times n$ matrices with band width $b$. Denoting their ensemble $\mathbb{B}_{n,b}$ according to \cite{MPK} when $\frac{b}n\xrightarrow[n,b\to\infty]{}0$ the eigenvalue distributions of $\mathbb{B}_{n,b}$ matrices normalized by $\frac1{\sqrt{b}}$ converge with all the moments to the semicircle law $\ds{\frac{\,1\,}{2\pi\s^2}\sqrt{4\s^2-x^2}\ \mathbb{I}_{[-2\s,2\s]}(x)}$, where $\mathbb{I}_S$ is the characteristic function of a set $S$. Therefore,
\begin{equation*}
\frac1{n\sqrt{b}}\,\E(\mathbb{B}_{n,b})\xrightarrow[n,b\to\infty]{}\frac{\,1\,}{2\pi\s^2}\int_{-2\s}^{2\s}|x|\sqrt{4\s^2-x^2}\,dx=\frac8{3\pi}\,\s\,,
\end{equation*}
and almost surely 
\begin{equation}\label{BnbAsEnrg}
\E(\mathbb{B}_{n,b})=n\sqrt{b}\left(\frac{\,8\,}{3\pi}\,\s+o(1)\right)\,.
\end{equation}
Thus, large random band symmetric T\"oeplitz matrices are more energetic on average than large general band symmetric  matrices of the same size and band width since $\frac{\,2\,}{\sqrt{\pi}}>\frac{\,8\,}{3\pi}$. When $\s=\sqrt{p\,(1-p)}$ this applies also to the corresponding random graphs. Interestingly enough for general symmetric T\"oeplitz matrices (without the band condition) the limiting distribution is neither Gaussian, nor semicircle, but a new one without an analytic expression for density and with sub-Gaussian even moments \cite{HM}. However, for general symmetric circulants the limit law is still Gaussian \cite{BM}.

Finally, we look at the energy of $d$--regular graphs. A graph is $d$--regular if each vertex has exactly $d$ edges adjacent to it \cite[1.4.1]{LSG}. The circulant graphs are always $d$--regular with $d=2m$, twice the number of jump sizes in the minimal representation of Definition \ref{circgraph}. It is shown in \cite{McK} using the method of moments that the spectral distributions of $d$--regular graphs converge to the Kesten law for $n\to\infty$ assuming that $\frac{c_k}n\to0$ for $k\geq3$, where $c_k$ is the number of cycles of length $k$ in the graph. 

The Kesten law density is $\frac{\,d\,}{2\pi}\frac{\sqrt{4(d-1)-x^2}}{d^2-x^2}\ \mathbb{I}_{[-2\sqrt{d-1},\,2\sqrt{d-1}]}(x)$. Denoting by $\R_{n,d}$ the ensemble of equally likely $d$--regular graphs on $n$ vertices we have, see also \cite{NikR}:
\begin{multline}\label{RndNmAsEnrg}
\frac1{n}\,\E(\R_{n,d})\xrightarrow[n\to\infty]{}\frac{\,d\,}{2\pi}\int_{-2\sqrt{d-1}}^{2\sqrt{d-1}}|x|\,\frac{\sqrt{4(d-1)-x^2}}{d^2-x^2}\,dx\\
=\frac{2d\,\sqrt{d-1}}{\pi}\left(1-\frac{d-2}{2\sqrt{d-1}}\ 
\tan^{-1}\frac{2\sqrt{d-1}}{d-2}\right)\,.
\end{multline}
Using the Taylor expansion of $\tan^{-1}(x)$ at $x=0$ one can see that for large $d$ the right hand side of 
\eqref{RndNmAsEnrg} is asymptotically equivalent to $\frac{\,8\,}{3\pi}\sqrt{d}$, hence 
\begin{equation}\label{RndAsEnrg}
\E(\R_{n,d})=n\sqrt{d}\left(\frac{\,8\,}{3\pi}+o(1)\right).
\end{equation}
The asymptotic is only valid, however, if $\frac{d^k}n\to0$ for all integer $k\geq3$ since the expected values of $c_k$ converge to $\frac{(d-1)^k}{2k}$ \cite{McK}. It should be compared to \eqref{Gnb1/2AsEnrg} because with high probablity $d\sim b$ in $\G_{n,b}\left(\frac12\right)$. Since $\frac1{\sqrt{\pi}}<\frac{\,8\,}{3\pi}$ general $d$--regular graphs are more energetic than the corresponding band circulants for relatively small $d$. Comparison to \eqref{TndAsEnrg} is inconclusive since $\frac1{\sqrt{3}}<\frac{\,8\,}{3\pi}<1$. Nikiforov shows in \cite{NikR} based on a result of \cite{TVW} that when both 
$d$ and $n-d$ tend to infinity with $n$ we have
\begin{equation}\label{RndAsEnrg}
\E(\R_{n,d})=n\sqrt{d(1-d/n)}\left(\frac{\,8\,}{3\pi}+o(1)\right).
\end{equation}
Comparing to \eqref{Gnp1/2AsEnrg} we see that the regular graphs with $d\sim n$ can be far less energetic than the general ones.

\section{Outline of proof}\label{S3}

Since the proofs of Theorems \ref{AnbThm}, \ref{LrgDevThm} involve many technicalities and cumbersome formulas we outline the main steps in this section. Namely, we sketch the computation of the normalized expected asymptotic energy of random band symmetric circulant matrices, and the estimation of its difference with the energies of finite matrices and of probabilities of large deviations. 

From \eqref{NormEnrgCirc} we have 
\begin{equation}\label{NbEnrgCirc}
\frac1{n\sqrt{b}}\,\EE\E(\A_{n,b})=\frac2{n\sqrt{b}}\,\sum_{r=1}^n\EE\left|\sum_{k=1}^b\xi_k\cos\Big(2\pi k\,\frac{r}n\Big)\right|\,,
\end{equation}
where $\xi_k$ are independent identically distributed random variables with mean $a$ and variance $\s^2$. We start by centering $\xi_k$ and splitting off the mean part of the sum under the absolute value sign. Set $\cxi_k:=\xi_k-a$, then the following estimate follows directly from \eqref{NbEnrgCirc}
\begin{multline}\label{CnEq}
\left|\frac1{n\sqrt{b}}\,\EE\E(\A_{n,b})-\frac2{n\sqrt{b}}\,\sum_{r=1}^n\EE\left|\sum_{k=1}^b\cxi_k\cos\Big(2\pi k\,\frac{r}n\Big)\right|\right|\\
\leq\frac{2|a|}{n\sqrt{b}}\,\sum_{r=1}^n\left|\sum_{k=1}^b\cos\Big(2\pi k\,\frac{r}n\Big)\right|\,.
\end{multline}
To estimate the right hand side of \eqref{CnEq} we notice that $\ds{\sum_{r=1}^n\left|\sum_{k=1}^b\cos\Big(2\pi k\,\frac{r}n\Big)\right|\frac1n}$ is a Riemann sum of a well-known function.
\begin{definition}\label{DirKer} The {\sl Dirichlet kernel} is the following function \cite[12.2]{Gold}
\begin{equation}\label{DirKer}
\D_b(t):=\sum_{k=-b}^be^{ikt}=1+2\sum_{k=1}^b\cos(kt)
=\begin{cases}\ds{\frac{\sin\Big(b+\frac12\Big)t}{\sin\Big(\frac{t}2\Big)}}, &t\neq0\\
2b+1, &t=0\,.\end{cases}
\end{equation}
\end{definition}
\noindent By inspection, the corresponding Riemann integral is $\ds{\int_0^1|\D_b(2\pi t)-1|\,dt}$, and it can be estimated via the so--called {\sl Lebesgue constants} $\ds{\L_b:=\frac1\pi\int_0^\pi|\D_b(t)|\,dt}$ \cite[4.2]{Finch}. 
It remains to bound the difference between the sum and the integral. A curious quirk is that we are using the integral to estimate the sums rather than the other way around, as is common in numerical analysis. The usual estimate in terms of Lipschitz constants is too rough for our purposes since the Lipschitz constant of the Dirichlet kernel grows as $b^2$. However, the estimate in terms of the total variation \cite[5.5]{DR} $\ds{\Big|\sum_{r=1}^nf\left(\frac{r}n\right)\frac1n-\int_0^1f(t)\,dt\Big|\leq\frac1n\textrm{Var}_{[0,1]}(f)}$ works better because, as we prove in Section \ref{S4}, the total variation of the Dirichlet kernel only grows as $b\ln b$\,. We were unable to find a suitable bound on the kernel's total variation in the existing literature despite the classical nature of the subject. Combining the estimates we get the following.
\begin{lemma}\label{MeanPart} With the notation above
\begin{multline}\label{MeanEst}
\left|\frac1{n\sqrt{b}}\,\EE\E(\A_{n,b})-\frac2{n\sqrt{b}}\,\sum_{r=1}^n\EE\left|\sum_{k=1}^b\cxi_k\cos\Big(2\pi k\,\frac{r}n\Big)\right|\right|\\
\leq\frac{\,2\,|a|\,}{\sqrt{b}}\left(2+\frac2{\pi^2}\ln b+\frac{1+(2b+1)(1+\gamma+\ln b))}{n}\right)
= O\left(\frac{\ln b}{\sqrt{b}}\right)\,,
\end{multline}
where $\gamma$ is the Euler constant.
\end{lemma}
Next we have to deal with the expected value of the centered sum $\ds{\EE\left|\sum_{k=1}^b\cxi_k\cos\Big(2\pi k\,\frac{r}n\Big)\right|}$\,. It is convenient to denote $\ds{\S_b(t):=\sum_{k=1}^b\cxi_k\cos\Big(2\pi k\,t\Big)}$, which is a weighted sum of independent identically distributed random variables with mean $0$ and finite variance $\s^2$\,. According to the central limit theorem (CLT) the normalized sums $\ds{\frac1{\sqrt{b}}\S_b(t)}$ converge in distribution to a Gaussian random variable as long as their variances have a limit and the Lindeberg condition is satisfied 
\cite[III.4.1]{Shir}. Verifying the conditions of CLT is the approach taken in \cite{Karg}. Since we are interested in  convergence of the first absolute moments with explicit rate estimates we use instead the method of \cite{Daug} of deriving moment estimates from the non-uniform Berry-Esseen bounds. An alternative approach to moment estimates was developed earlier in \cite{vBr}, but the resulting inequalities are less precise.

Let $X_k$ be a sequence of independent(not necessarily identically distributed) random variables with finite variances and set $\SS_b:=\sum_{k=1}^bX_k$, $B_b:=\sum_{k=1}^b\EE|X_k|^2$. Also denote $\NN_{a,\s}$ a Gaussian random variable with mean $a$ and variance $\s^2$. The estimate we use is 
\begin{equation}\label{SumGauss}
\Big|\EE|\SS_b|-\EE|\NN_{0,\sqrt{B_b}}|\Big|\leq\frac{2\pi C_1}{3\sqrt{3}}\frac{\sum_{k=1}^b\EE|X_k|^3}{\sum_{k=1}^b\EE|X_k|^2}\,,
\end{equation}
where $C_1$ is the best constant from the non-uniform Berry--Esseen inequality. The best current estimate is 
$C_1<31.954$ \cite{Pad}, but it is expected that the actual value of $C_1$ is smaller by an order of magnitude \cite{Pin}\,. In our case $X_k=\cxi_k\cos(2\pi kt)$, so that $\SS_b=\S_b(t)$ and $B_b=\s^2\sum_{k=1}^b\cos^2(2\pi kt)$\,. Denoting $\mu_3:=\EE|\cxi_k|^3$ one can see that $\ds{\frac{\sum_{k=1}^b\EE|X_k|^3}{\sum_{k=1}^b\EE|X_k|^2}\leq\frac{\mu_3}{\s^2}}$. Therefore  we obtain the following.
\begin{lemma}\label{CLTPart} Assume that $\xi_k$ and hence $\cxi_k$ have finite third moments. Then in the notation above for any $t\in[0,1]$:
\begin{equation}\label{CLTEst}
\left|\EE\left|\sum_{k=1}^b\cxi_k\cos\Big(2\pi k\,t\Big)\right|-\EE\left|\NN_{0,\,\s\sqrt{\sum_{k=1}^b\cos^2\left(2\pi k\,t\right)}}\right|\right|\leq\frac{4\pi C_1}{3\sqrt{3}}\,\frac{\mu_3}{\s^2}\,,\\
\end{equation}
where $C_1<31.954$\,.
\end{lemma}
\noindent By direct calculation $\EE|\NN_{0,\tau}|=\tau\,\EE|\NN_{0,1}|=\sqrt{\frac2\pi}\,\tau$. 
Taking $\tau=\s\sqrt{\sum_{k=1}^b\cos^2\left(2\pi k\,t\right)}$ we see from \eqref{MeanEst} and \eqref{CLTEst} 
that $\frac1{n\sqrt{b}}\,\EE\E(\A_{n,b})$ has the same limit when $n,b\to\infty$ as 
\begin{equation}\label{RootSum}
\frac2{n\sqrt{b}}\sum_{r=1}^n\sqrt{\frac2\pi}\,\s\,\sqrt{\sum_{k=1}^b\cos^2\left(2\pi k\,\frac{r}n\right)}
=\frac{\,2\sqrt{2}\,\s\,}{\sqrt{\pi}}\sum_{r=1}^n
\sqrt{\frac1b\sum_{k=1}^b\cos^2\left(2\pi k\,\frac{r}n\right)}\,\,\frac1n\,.
\end{equation}
It remains to find this limit and estimate the rate of convergence to it. The right hand side of \eqref{RootSum} is again a Riemann sum and in the integrand 
$$
\frac1b\sum_{k=1}^b\cos^2\left(2\pi k\,t\right)=\frac12+\frac1{2b}\sum_{k=1}^b\cos^2\left(2\pi k\,t\right)\xrightarrow[b\to\infty]{}\frac12\,,
$$ 
where the convergence is for all $t\in[0,1]$ except $t=0,\frac12,1$. Therefore the limit, if it exists, must be equal to $\ds{\frac{\,2\sqrt{2}\,\s\,}{\sqrt{\pi}}\int_0^1\frac1{\sqrt{2}}\,dt=\frac{\,2\,\s\,}{\sqrt{\pi}}}$.
To prove existence of the limit and to bound the difference the Dirichlet kernel estimates can be used as in 
Lemma \ref{MeanPart} since 
$$
\frac1b\sum_{k=1}^b\cos^2\left(2\pi k\,t\right)=\frac1{2b}+\frac12\sum_{k=1}^b\cos\left(4\pi k\,t\right)
=\frac1{2b}+\frac14(\D_b(4\pi t)-1),
$$
where $\D_b$ is the Dirichlet kernel \eqref{DirKer}. 
\begin{lemma}\label{VarPart} With the notation above
\begin{multline}\label{VarEst}
\left|\frac2{n\sqrt{b}}\,\sum_{r=1}^n\EE\left|\NN_{0,\,\s\sqrt{\sum_{k=1}^b\cos^2\left(2\pi k\,t\right)}}\right|
-\frac{\,2\,\s\,}{\sqrt{\pi}}\right|\\
\leq\frac{\,4\,\s\,}{b\sqrt{\pi}}\left(1+\frac1{\pi^2}\ln b+\frac{1+(2b+1)(1+\gamma+\ln b)}{n}\right)
= O\left(\frac{\ln b}{\sqrt{b}}\right)\,,
\end{multline}
where $\gamma$ is the Euler constant.
\end{lemma}
\noindent Theorem \ref{AnbThm} follows by inspection from Lemmas \ref{MeanPart}--\ref{VarPart}. 

The proof of Theorem \ref{LrgDevThm} is based on an inequality of Talagrand. Let $F:[0,R]^b\to\R$ be a convex Lipschitz function with the Lipschitz constant $L_F$ and median $M_F$. Then for any product measure $\mathbb{P}$ induced by a probability measure on $[0,R]$ one has \cite[Thm.6.6]{Tal}:
\begin{equation}\label{TalaBase}
\mathbb{P}\{|F(x)-M_F|\geq\delta\}\leq4\,e^{-\frac{\delta^2}{4R^2L_F^2}}\,.
\end{equation}
Denoting $A_{n,b}(x):=\text{\rm Circ}(0,x_1,\dots,x_b,0,\dots,0,x_b,\dots,x_1)$ we apply this inequality to
\begin{equation}\label{TalaFunc}
F(x):=\frac1{n\sqrt{b}}\,\E\Big(A_{n,b}(x)\Big)=\frac1{n\sqrt{b}}\,{\rm tr}(|A_{n,b}(x)|)
=\frac2{n\sqrt{b}}\,\sum_{r=1}^n\left|\sum_{k=1}^bx_k\cos\Big(2\pi k\,\frac{r}n\Big)\right|\,,
\end{equation}
where the last expression follows from \eqref{EnrgCirc}. It is clear from \eqref{TalaFunc} that $F$ is convex and Lipschitz. One can show that $L_F\leq\sqrt{\frac2b}$ based on \cite[Lem 1.2b]{GZ}, and that $|\EE F-M_F|\leq4\sqrt{\frac{2\pi R^2}b}$ leading to the estimate in Theorem \ref{LrgDevThm}.

\section{Dirichlet kernel}\label{S4}

In this section we study some properties of the Dirichlet kernel $\D_b(t)$ \eqref{DirKer}, and prove Lemmas 
\ref{MeanPart} and \ref{VarPart} that depend on them. Recall from Section \ref{S3} that the kernel appears in two different contexts in our proofs, first when bounding the differences in the means, and later in the variances. The Riemann sums of $|\D_b(\pi mt)-1|$ appear in those cases with $m=2$ and $m=4$ respectively. Our strategy for dealing with the Riemann sums is to bound the integral, and then to estimate the difference between the sum and the integral using the total variation. 

Let $\textrm{Var}_{[a,b]}(f)$ denote the total variation of $f$ on $[a,b]$. First we reduce the estimates for $|\D_b(\pi mt)-1|$ to those for the Dirichlet kernel itself.
\begin{lemma}\label{Dbpimt} For any $m\in\Z\backslash\{0\}$:

{\rm(i)} $\ds{\int_0^1|\D_b(\pi mt)-1|\,dt=1+\frac1\pi\int_0^\pi|\D_b(t)|\,dt}$\,;

{\rm(ii)} $\ds{{\rm Var}_{\,[0,1]}\,\big(|\D_b(\pi mt)-1|\big)=|m|{\rm Var}_{\,[0,\pi]}(\D_b)}$\,.
\end{lemma}
\begin{proof} Since $\D_b(t)$ is even $\D_b(t)=\D_b(-t)$, and we may assume $m>0$ without loss of generality. For the same reason integrals with $\D_b(t)$ over $[-\pi,0]$ and $[0,\pi]$ are equal and, since it is also $2\pi$ periodic, integrals with it over any interval $[\pi k,\pi(k+1)]$ with $k\in\Z$ are equal. Therefore,
$$
\ds{\int_0^1|\D_b(\pi mt)-1|\,dt\leq 1+\int_0^1|\D_b(\pi mt)|\,dt
=1+\frac1{\pi m}\int_0^{\pi m}|\D_b(t)|\,dt=1+\frac1{\pi}\int_0^{\pi}|\D_b(t)|\,dt\,.}
$$
For (ii) note that taking the absolute value and subtracting a constant does not change the total variation, while $\ds{{\rm Var}_{\,[0,1]}\,\big(\D_b(\pi mt)\big)={\rm Var}_{\,[0,\pi m]}\,\big(\D_b\big)}$. As with the integrals above the total variations over all intervals $[\pi k,\pi(k+1)]$ are the same, so the last expression is equal to $m{\rm Var}_{\,[0,\pi]}\,\big(\D_b\big)$\,.
\end{proof}
The numbers $\L_b:=\ds{\frac1\pi\int_0^\pi|\D_b(t)|\,dt}$ are known as the {\sl Lebesgue constants} and satisfy 
$\ds{\L_b\leq3+\frac4{\pi^2}\ln b}$\ \cite[4.2]{Finch}. The difference between Riemann sums and integrals satisfies $\ds{\Big|\sum_{r=1}^nf\left(\frac{r}n\right)\frac1n-\int_0^1f(t)\,dt\Big|\leq\frac1n\,\textrm{Var}_{[0,1]}(f)}$ 
\cite[5.5]{DR}, and to complete the estimate we need to bound the total variation of the Dirichlet kernel. Although the Dirichlet kernel is a classical function we were unable to find suitable estimates of its total variation in the literature.
\begin{lemma}\label{DbTotVar} Let $\gamma$ denote the Euler constant and $b\in\N$, then

\center $\ds{{\rm Var}_{\,[0,\pi]}(\D_b)\leq1+(2b+1)(1+\gamma+\ln b)}$\,.
\end{lemma}
\begin{proof} It is convenient to set $m:=2b+1$ and work with $D_m(t):=\D_b(2t)$ 
since $D_m(t)=\frac{\sin(mt)}{\sin t}$ for $t\neq0$, while $\ds{{\rm Var}_{\,[0,\pi]}(\D_b)={\rm Var}_{\,[0,\frac{\pi}2]}(D_m)}$. Consider $D_m(t)$ on $[-\frac{\pi}2,\frac{\pi}2]$ first, which is its period since $m$ is odd. Its zeros on this interval are the zeros of the numerator except for $t=0$, where $D_m(0)=m$, namely $t=\frac{\pi k}m$, $k=\pm1,\dots,\pm\frac{m-1}2$\,. Between any two consecutive zeros there must be at least one local extremum for the total of at least $2\big(\frac{m-1}2-1\big)+1=m-2$.
But $D_m(t)$ is a polynomial of degree $m-1$ in $\sin t$, that can have no more than $m-2$ local extrema on the entire real axis. Since the derivative of $\sin t$ on $(-\frac{\pi}2,\frac{\pi}2)$ is strictly positive $D_m(t)$ has no more than $m-2$ local extrema in the interior, and therefore exactly $m-2$ interior extrema. Since $D_m(t)$ is even $t=0$ must be one of them, and since it is $\pi$ periodic the endpoints $t=\pm\frac{\pi}2$ are also local extrema.

Restricting to $[0,\frac{\pi}2]$ we see that there are extrema at $t=0,\frac{\pi}2$ and exactly one on each 
interval $[\frac{\pi k}m,\frac{\pi(k+1)}m]$ for $k=1,\dots,\frac{m-3}2$. On the boundary intervals $[0,\frac{\pi}m]$ and 
$[\frac{\pi(m-1)}{2m},\frac{\pi}2]$ the variations are $D_m(0)=m$ and $D_m\big(\frac{\pi}2\big)=1$ respectively (from the extremum to $0$), and on the remaining internal intervals they are $2\max\{\,|D_m(t)|\,\Big|\,t\in\big[\frac{\pi k}m,\frac{\pi(k+1)}m\big]\}$ (from $0$ to the extremum and back). Since for $t\in[\frac{\pi k}m,\frac{\pi(k+1)}m]$:
$
\ds{\left|\frac{\sin(mt)}{\sin t}\right|\leq\frac1{\sin t}\leq\frac1{\sin\frac{\pi k}m}
\leq\frac1{\frac2{\pi}\frac{\pi k}m}=\frac{m}{2k}}
$
we conclude 
\begin{multline}{\rm Var}_{\,[0,\frac{\pi}2]}(D_m)=D_m(0)
+2\sum_{k=1}^{\frac{m-1}{3}}\max\left\{\,|D_m(t)|\,\Big|\,t\in\left[\frac{\pi k}m,\frac{\pi(k+1)}m\right]\right\}
+D_m\Big(\frac{\pi}2\Big)\\
\leq m+m\sum_{k=1}^{\frac{m-1}{3}}\frac1k+1=(2b+1)\left(1+\sum_{k=1}^{\frac{m-1}{3}}\frac1k\right)+1
\leq(2b+1)\left(1+\gamma+\ln b\right)+1\,.
\end{multline}
In the last inequality we used a standard estimate for the partial sums of the harmonic series.
\end{proof} 
We are now ready to prove Lemmas \ref{MeanPart} and \ref{VarPart}\,.
\begin{proof}[Proof of {\rm\bf Lemma \ref{MeanPart}}] Starting with \eqref{CnEq} by definition of $\D_b(t)$ and Lemmas \ref{Dbpimt}, \ref{DbTotVar} we have:
\begin{align*}\label{MeanEst}
\frac{2}{n}\,&\sum_{r=1}^n\left|\sum_{k=1}^b\cos\Big(2\pi k\,\frac{r}n\Big)\right|
=\sum_{r=1}^n\left|\D_b\Big(2\pi\,\frac{r}{n}\Big)-1\right|\frac1n\\
&\leq\int_0^1|\D_b(2\pi t)-1|\,dt+\frac1n\,{\rm Var}_{\,[0,1]}\,\big(|\D_b(2\pi t)-1|\big)\\
&\leq1+\L_b+\frac2n\,{\rm Var}_{\,[0,\pi]}\,\big(\D_b\big)
\leq4\left(1+\frac1{\pi^2}\ln b\right)+\frac2n\Big((2b+1)(1+\gamma+\ln b)+1\Big)\,,
\end{align*}
where we used the estimate $\ds{\L_b\leq3+\frac4{\pi^2}\ln b}$\ \cite[4.2]{Finch} for the Lebesgue constants $\L_b$.
Since $b<\frac{n}2$ the last expression is $O\left(\ln b\right)$\,, and it remains to divide both sides by $\sqrt{b}$.
\end{proof}

\begin{proof}[Proof of {\rm\bf Lemma \ref{VarPart}}] We have:
\begin{align*}
&\left|\frac2{n\sqrt{b}}\,\sum_{r=1}^n\EE\left|\NN_{0,\,\s\sqrt{\sum_{k=1}^b\cos^2\left(2\pi k\,\frac{r}n\right)}}\right|
-\frac{\,2\,\s\,}{\sqrt{\pi}}\right|
=\left|\frac2{n\sqrt{b}}\,\sum_{r=1}^n\sqrt{\frac2\pi}\,\s\sqrt{\sum_{k=1}^b\cos^2\left(2\pi k\,\frac{r}n\right)}
-\frac{\,2\,\s\,}{\sqrt{\pi}}\right|\\
&\leq\frac{\,2\,\s\,}{n\sqrt{\pi}}\sum_{r=1}^n\left|\sqrt{\frac2b\sum_{k=1}^b\cos^2\left(2\pi k\,\frac{r}n\right)}-1\right|
=\frac{\,2\,\s\,}{\sqrt{\pi}}\sum_{r=1}^n\,\frac{\left|\frac2b\sum_{k=1}^b\cos^2\left(2\pi k\,\frac{r}n\right)-1\right|}{1+\sqrt{\frac2b\sum_{k=1}^b\cos^2\left(2\pi k\,\frac{r}n\right)}}\\
&\leq\frac{\,2\,\s\,}{n\,\sqrt{\pi}}\sum_{r=1}^n\left|\frac2b\left(\frac{b}2+\frac12\sum_{k=1}^b\cos\left(4\pi k\,\frac{r}n\right)\right)-1\right|
=\frac{\,2\,\s\,}{nb\,\sqrt{\pi}}\sum_{r=1}^n\left|\sum_{k=1}^b\cos\left(4\pi k\,\frac{r}n\right)\right|.
\end{align*}
As in the proof of Proof of Lemma \ref{MeanPart} the last expression is bounded by 
$$
\frac{\,\s\,}{b\,\sqrt{\pi}}\left(1+\L_b+\frac4n\,{\rm Var}_{\,[0,\pi]}\,\big(\D_b\big)\right)
\leq\frac{\,4\,\s\,}{b\,\sqrt{\pi}}\left(1+\frac1{\pi^2}\ln b+\frac1n\Big((2b+1)(1+\gamma+\ln b)+1\Big)\right)\!,
$$
which is $\ds{O\left(\frac{\ln b}b\right)}$ since $b<\frac{n}2$\,.
\end{proof}

\section{Berry-Esseen bound}\label{S5}

In this section we derive an estimate for the difference between the first absolute moments of sums of independent non--identically distributed random variables and their Gaussian limits required to prove Lemma \ref{CLTPart}. Our starting point is a non--uniform version of the Berry-Esseen inequality, which lends itself nicely to estimating moments.

Let $X_k$ be a sequence of independent not necessarily identically distributed centered random variables with finite variances and finite third absolute moments. Denote $\mathcal{S}_b:=\sum_{k=1}^bX_k$, $B_b:=\sum_{k=1}^b\EE|X_k|^2$.
Let $\Phi_{a,\s}$ be the distribution function of a Gaussian random variable $\NN_{a,\s}$ with mean $a$ and variance 
$\s^2$. A non-uniform Berry-Esseen inequality is \cite{Pad}:
\begin{equation*}
\left|\P\left\{\frac1{\sqrt{B_b}}\,\mathcal{S}_b\leq x\right\}-\Phi_{0,1}(x)\right|
\leq C_1\,\frac{\sum_{k=1}^b\EE|X_k|^3}{B_b^{\frac32}}\,\frac1{1+|x|^3},
\end{equation*}
or after rescaling
\begin{equation}
\left|\P\left\{\mathcal{S}_b\leq x\right\}-\Phi_{0,\sqrt{B_b}}(x)\right|
\leq C_1\,\frac{\sum_{k=1}^b\EE|X_k|^3}{B_b^{\frac32}}\,\frac1{1+\left|\frac{x}{\sqrt{B_b}}\right|^3},
\end{equation}
where $C_1<31.935$ is an absolute constant \cite{Pad}. The idea of estimating moments via non-uniform Berry-Esseen inequalities goes back at least to \cite{Daug}.
\begin{lemma}\label{LmSumGauss} With the notation above:
\begin{equation}\label{SumGauss2} 
\Big|\EE|\SS_b|-\EE|\NN_{0,\sqrt{B_b}}|\Big|\leq\frac{4\pi C_1}{3\sqrt{3}}\,
\frac{\sum_{k=1}^b\EE|X_k|^3}{B_b}\,,
\end{equation}
\end{lemma}
\begin{proof} Let $F(x)$ be an integrable function of bounded variation on $\R$ satisfying $|x|F(x)\xrightarrow[|x|\to\infty]{}0$\,. Integrating by parts in Lebesgue--Stiltjes integrals \cite[II.6]{Shir},
\begin{equation*}
\Big|\int_{\R}|x|\,dF(x)\Big|=\left|\,|x|F(x)\Big|_{-\infty}^\infty-\int_{\R}F(x)\,d|x|\,\right|
\leq\int_{\R}|F(x)|\,\big|d|x|\big|=\int_{\R}|F(x)|\,dx\,.
\end{equation*}
Applying this to $F(x)=\P\left\{\mathcal{S}_b\leq x\right\}-\Phi_{0,\sqrt{B_b}}(x)$ we have 
\begin{multline*}
\Big|\EE|\SS_b|-\EE|\NN_{0,\sqrt{B_b}}|\Big|\leq\Big|\int_{\R}|x|\,dF(x)\Big|
\leq\int_{\R}|F(x)|\,dx\\
\leq C_1\,\frac{\sum_{k=1}^b\EE|X_k|^3}{B_b^{\frac32}}\int_{\R}\frac{dx}{1+\left|\frac{x}{\sqrt{B_b}}\right|^3}
\leq C_1\,\frac{\sum_{k=1}^b\EE|X_k|^3}{B_b^{\frac32}}\,\sqrt{B_b}\,\int_{\R}\frac{dx}{1+\left|x\right|^3}\,.
\end{multline*}
Finally, by an elementary computation $\ds{\int_{\R}\frac{dx}{1+\left|x\right|^3}}=\frac{4\pi}{3\sqrt{3}}$\,.
\end{proof}
\noindent Our Berry-Esseen bound \eqref{CLTEst} is a direct application of the preceeding Lemma.
\begin{proof}[Proof of {\rm\bf Lemma \ref{CLTPart}}]
We take $X_k=\cxi_k\cos(2\pi k\frac{r}{n})$, where $\cxi_k$ are centered independent identically distributed random variables with $\EE|\cxi_k|^2=\s^2$ and $\EE|\cxi_k|^3=\mu_3$, so that 
$\mathcal{S}_b:=\sum_{k=1}^b\cxi_k\cos(2\pi k\frac{r}{n})$, $B_b=\s^2\sum_{k=1}^b|\cos(2\pi k\frac{r}{n})|^2$ and 
$\sum_{k=1}^b\EE|X_k|^3=\mu_3\sum_{k=1}^b|\cos(2\pi k\frac{r}{n})|^3$. By Lemma \ref{LmSumGauss},
\begin{multline}\label{BerEs} 
\left|\EE\left|\sum_{k=1}^b\cxi_k\cos\Big(2\pi k\,t\Big)\right|-\EE\left|\NN_{0,\,\s\sqrt{\sum_{k=1}^b\cos^2\left(2\pi k\,t\right)}}\right|\right|
\leq\frac{4\pi C_1}{3\sqrt{3}}\,\frac{\sum_{k=1}^b\EE|X_k|^3}{\sum_{k=1}^b\EE|X_k|^2}\\
=\frac{4\pi C_1}{3\sqrt{3}}\ \frac{\mu_3\sum_{k=1}^b|\cos(2\pi k\frac{r}{n})|^3}{\s^2\sum_{k=1}^b|\cos(2\pi k\frac{r}{n})|^2}
\leq\frac{4\pi C_1}{3\sqrt{3}}\ \frac{\mu_3}{\s^2}\ \frac{\max\limits_{1\leq k\leq b}|\cos(2\pi k\frac{r}{n})|\,\sum_{k=1}^b|\cos(2\pi k\frac{r}{n})|^2}{\sum_{k=1}^b|\cos(2\pi k\frac{r}{n})|^2}
\leq\frac{4\pi C_1}{3\sqrt{3}}\ \frac{\mu_3}{\s^2}\,.
\end{multline}
\end{proof}

\section{Large deviations}\label{S6}

Assuming that random variables in our ensembles have finite moments one can only derive polynomial bounds on the probabilities of large deviations. In view of our applications to graph theory, where the variables are Bernoullian, we prefer to assume that their distributions are compactly supported and derive exponential bounds instead. We essentially follow the approach of \cite{GZ}, who in their turn rely on a concentration inequality of Talagrand for product measures.

Assume that $\xi_k$ take values in a finite interval $[0,R]$ of length $R$. Consider a convex Lipschitz function $F:[0,R]^b\to\R$ with the Lipschitz constant $L_F$ and median $M_F:=\sup\Big\{t\geq0\,\big|\,\mathbb{P}\{F(x)\leq t\}\leq\frac12\Big\}$. Then for any probability measure supported on $[0,R]$ one has 
\begin{equation}\label{TalaComp} 
\ds{\mathbb{P}\{|F(x)-M_F|\geq\delta\}\leq4\,e^{-\frac{\delta^2}{4R^2L_F^2}}}\,,
\end{equation}
where $\mathbb{P}$ is the induced product measure on $[0,R]^b$\,, see \cite[Thm.6.6]{Tal}. There is also a version of Talagrand inequalities for variables with distributions satisfying a log--Sobolev inequality \cite{GZ}. For variables with finite moments only polynomial bounds can be derived for probabilities of large deviations, 
see e.g. in \cite[6.1]{HM}. 

To use \eqref{TalaComp} in our case we need a good estimate for the Lipschitz constant of 
\begin{equation}\label{TalaTrFunc}
F(x):=\frac1{n\sqrt{b}}\,\E\Big(A_{n,b}(x)\Big)=\frac1{n\sqrt{b}}\,{\rm tr}(|A_{n,b}(x)|)
=\frac2{n\sqrt{b}}\,\sum_{r=1}^n\left|\sum_{k=1}^bx_k\cos\Big(2\pi k\,\frac{r}n\Big)\right|\,,
\end{equation}
where $A_{n,b}(x):=\text{\rm Circ}(0,x_1,\dots,x_b,0,\dots,0,x_b,\dots,x_1)$. To this end, one may be tempted to apply the Cauchy-Schwarz to the last expression, but the resulting bound is not very good. Instead consider functions of the form ${\rm tr}f(A)$, where $f$ is a differentiable real valued function on $\R$ with the uniformly bounded derivative, and $A=(a_{ij})$ is a symmetric $n\times n$ real matrix. Treating ${\rm tr}f(A)$ as a function of 
$\frac{n(n+1)}2$ variables $a_{ij}$ for $i\leq j$, the proof of Lemma 1.2b in \cite{GZ} implies that
\begin{equation}\label{trder} 
\sum\limits_{1\leq i\leq j\leq n}\left|\frac{\d\,{\rm tr}f(A)}{\d a_{ij}}\right|^2\leq 2n\|f'\|_\infty^2\,,
\end{equation}
where $\|f'\|_\infty:=\sup\limits_{x\in\R}|f'(x)|$. Note that $\|f'\|_\infty<\infty$ implies that $f$ is Lipschitz with the Lipschitz constant $L_f=\|f'\|_\infty$. By \eqref{trder} the gradient of ${\rm tr}f$ is also uniformly bounded on $\R^{\frac{n(n+1)}2}$, and therefore ${\rm tr}f$ is also Lipschitz with 
$L_{{\rm tr}f}\leq\sqrt{2n}\,\|f'\|_\infty$. Hence for any symmetric real matrices $A$ and $B$:
\begin{equation}\label{trLip} 
\Big|{\rm tr}f(A)-{\rm tr}f(B)\Big|\leq\sqrt{2n}\,\|f'\|_\infty\,\left(\sum\limits_{1\leq i\leq j\leq n}|a_{ij}-b_{ij}|^2\right)^\frac12\,.
\end{equation}
We can not apply \eqref{trLip} to $f(x)=|x|$ directly because it is not differentiable, but a limit argument succeeds.
\begin{lemma}\label{Liptrabs} Let $A=(a_{ij})$ and $B=(b_{ij})$ be any symmetric real matrices, then
\begin{equation}\label{trLipabs} 
\Big|{\rm tr}|A|-{\rm tr}|B|\Big|\leq\sqrt{2n}\,\left(\sum\limits_{1\leq i\leq j\leq n}|a_{ij}-b_{ij}|^2\right)^\frac12\,.
\end{equation}
\end{lemma}
\begin{proof} Although Lipschitz functions can not always be approximated by differentiable functions in the Lipschitz norm, $f(x)=|x|$ can be uniformly approximated on $\R$ by $C^\infty$ functions with Lipschitz constants less than or equal to $1$. For instance, by $f_n(x)=\frac1n\sqrt{1+n^2x^2}$. Since $\|f'_n\|_\infty=L_{f_n}\leq1$ and the convergence is unifrom we can pass to limit in \eqref{trLip} to get \eqref{trLipabs}.
\end{proof}
\begin{lemma}\label{LipEnrg} For $\ds{F(x)=\frac1{n\sqrt{b}}\,\E\Big(A_{n,b}(x)\Big)}$ the Lipschitz constant satisfies $L_F\leq\sqrt{\frac2b}$\,.
\end{lemma}
\begin{proof} Note that if $A,B$ are symmetric with zeros on the main diagonal then
$$
\ds{\sum\limits_{1\leq i\leq j\leq n}|a_{ij}-b_{ij}|^2=\frac12\sum\limits_{i,j=1}^n|a_{ij}-b_{ij}|^2}.
$$
Applying Lemma \ref{Liptrabs}:
\begin{multline*}
\Big|{\rm tr}|A_{n,b}(x)|-{\rm tr}|A_{n,b}(y)|\Big|
\leq\sqrt{2n}\left(\frac12\sum\limits_{i,j=1}^n|a_{ij}-b_{ij}|^2\right)^\frac12\\
=\sqrt{2n}\left(\sum\limits_{i=1}^n\sum\limits_{k=1}^b|x_k-y_k|^2\right)^\frac12
=\sqrt{2n}\sqrt{n}\left(\sum\limits_{k=1}^b|x_k-y_k|^2\right)^\frac12
=n\sqrt{2}|x-y|\,.
\end{multline*}
It remains to divide both sides by $n\sqrt{b}$\,.
\end{proof}
Lemma \ref{LipEnrg} based on inequality \eqref{trLipabs} provides a much better bound on the Lipschitz constant of the normalized energy function than the direct Cauchy-Schwarz estimate, which only gives $L_F\leq\sqrt{2}$. We are now ready to prove the main result.
\begin{proof}[Proof of {\rm\bf Theorem \ref{LrgDevThm}}]
Applying the Talagrand inequality \eqref{TalaComp} to $\ds{F(x)=\frac1{n\sqrt{b}}A_{n,b}(x)}$ we have
\begin{equation}\label{TalaSub}
\mathbb{P}\{|F(x)-M_F|\geq\delta\}\leq4\,e^{-\frac{\delta^2}{4R^2\left(\sqrt{\frac2b}\right)^2}}
=4\,e^{-\frac{b}{8R^2}\delta^2}\,.
\end{equation}
Furthermore, 
\begin{multline*}
|\EE F-M_F|\leq\EE|F-M_F|=\int_0^\infty\mathbb{P}\{|F(x)-M_F|\geq\delta\}\,d\delta\\
\leq4\int_0^\infty e^{-\frac{b}{8R^2}\delta^2}\,d\delta=4\sqrt{\frac{2\pi R^2}{b}}=\delta_0(b)\,.
\end{multline*}
Since $\delta\leq|F(x)-\EE F|\leq|F(x)-M_F|+\delta_0(b)$ implies $|F(x)-M_F|\geq\delta-\delta_0(b)$ the desired estimate follows directly from \eqref{TalaSub}.
\end{proof}

\section{T\"oeplitz matrices}\label{S7}

In this section we prove Theorem \ref{ToepCirc}, which estimates the difference between normalized expected energies of symmetric band T\"oeplitz and circulant matrices. The estimate is based on the corner trick that turns a band T\"oeplitz matrix into a band circulant by altering it in the upper right and the lower left corners \cite[4.3]{Gray}.

Let $\widetilde{A}:=\text{\rm T\"oep}_s(a_0,a_1,\dots,a_b,0\dots,0)$ be a real symmetric $n\times n$ band T\"oeplitz matrix of band width $b<\frac{n}2$. Then $A:=\text{\rm Circ}(a_0,a_1,\dots,a_b,0,\dots,0,a_{b},\dots,a_{1})$ is a circulant of the same band width. One can see that the difference has the block structure
\begin{equation}\label{DifToepCirc} 
A-\widetilde{A}=\left(
\begin{array}{c|c|c}0&0&\LL\\ \hline
0&0&0\\ \hline
\LL^*&0&0\end{array}\right),
\end{equation}
where 
\begin{equation}\label{CornDif} 
\LL:=\text{\rm T\"oep}(0,\dots,0,a_b,\dots,a_1)=\begin{pmatrix}a_b&a_{b-1}&\dots&a_2&a_1\\
0&a_b&\dots&a_3&a_2\\
\vdots&\vdots&\ddots&\vdots&\vdots\\
0&0&\dots&a_b&a_{b-1}\\
0&0&\dots&0&a_b\\
\end{pmatrix}.
\end{equation}
When $n$ is large and $b\ll n$ the diffrence is relatively small and the normalized spectrum of 
$\widetilde{A}$ is well approximated by the normalized spectrum of $A$. We will estimate the trace norm of the difference.
\begin{lemma}\label{AltEst} In the notation above ${\rm tr}|A-\widetilde{A}|\leq2\sqrt{b}\,\Big(\sum_{k=1}^bk\,a_k^2\Big)^\frac12$\,.
\end{lemma}
\begin{proof}
By definition of the trace norm and \eqref{DifToepCirc},
\begin{multline}\label{TrDifToepCirc} 
{\rm tr}|A-\widetilde{A}|={\rm tr}\Big((A-\widetilde{A})^*(A-\widetilde{A})\Big)^\frac12\\
={\rm tr}\left(\begin{array}{c|c|c}\LL\LL^*&0&0\\ \hline
0&0&0\\ \hline
0&0&\LL^*\LL\end{array}\right)^{\!\!\frac12}
={\rm tr}(\LL\LL^*)^\frac12+{\rm tr}(\LL^*\LL)^\frac12=2\,{\rm tr}(\LL^*\LL)^\frac12
\end{multline}
since the eigenvalues of $(\LL\LL^*)^\frac12$ and $(\LL^*\LL)^\frac12$ are the same, namely the singular values 
$s_k$ of $\LL$. By the Cauchy-Schwarz, 
\begin{equation}\label{TrAbsEst} 
{\rm tr}(\LL^*\LL)^\frac12=\sum_{k=1}^bs_k\leq\Big(\sum_{k=1}^bs_k^2\Big)^\frac12\Big(\sum_{k=1}^b1^2\Big)^\frac12
=({\rm tr}\LL^*\LL)^\frac12\,\sqrt{b}.
\end{equation}
By inspection, the diagonal entries of $\LL^*\LL$ are $a_b^2$, $a_b^2+a_{b-1}^2$,\dots, so 
${\rm tr}\LL^*\LL=\sum_{k=1}^bk\,a_k^2$.
\end{proof}
\begin{proof}[Proof of {\rm\bf Theorem \ref{ToepCirc}}]
The corner trick gives a one-to-one probability preserving correspondence between ensembles 
$\widetilde{\A}_{n,b}$ and $\A_{n,b}$. Recall that we use the same notation to denote random elements of the corresponding ensembles. For the duration of this proof however we assume that $\widetilde{\A}_{n,b}$ is chosen randomly but $\A_{n,b}$ is obtained from it by the corner trick. In particular, $\widetilde{\A}_{n,b}$ and $\A_{n,b}$ are defined on the same probability space. With this in mind,
\begin{multline}\label{ExpAbsDif} 
\frac1{n\sqrt{b}}\,\left|\EE\E(\widetilde{\A}_{n,b})-\EE\E(\A_{n,b})\right|
\leq\frac1{n\sqrt{b}}\,\EE\left|\E(\widetilde{\A}_{n,b})-\E(\A_{n,b})\right|
=\frac1{n\sqrt{b}}\,\EE\,\Big|{\rm tr}|\widetilde{\A}_{n,b}|-{\rm tr}|\A_{n,b}|\Big|\\
\leq\frac1{n\sqrt{b}}\ \EE\,{\rm tr}\Big|\widetilde{\A}_{n,b}-\A_{n,b}\Big|
\leq\frac2n\,\EE\Big(\sum_{k=1}^bk\,\xi_k^2\Big)^\frac12,
\end{multline}
where $\xi_k$ are the independent random variables from the definitions of $\widetilde{\A}_{n,b}$ and $\A_{n,b}$, and the last inequality follows from Lemma \ref{AltEst}. Since for positive $x$ the function $x^\frac12$ is concave down  for any positive random variable $\xi$ we have by the Jensen inequality \cite[II.6.5]{Shir} that 
$\EE\xi^\frac12\leq\big(\EE\xi\big)^\frac12$. By the Minkowski inequality also $\EE\xi_k^2\leq(|a|+\s)^2$, where $a,\s$ are the mean and the variance of $\xi_k$ respectively. Therefore,
\begin{equation*}
\EE\Big(\sum_{k=1}^bk\,\xi_k^2\Big)^\frac12
\leq\left(\sum_{k=1}^bk\,\EE\xi_k^2\right)^{\frac12}
\leq\left(\EE\xi_k^2\right)^{\frac12}\left(\sum_{k=1}^bk\right)^{\frac12}
=\left(\EE\xi_k^2\right)^{\frac12}\,\sqrt{\frac{b(b+1)}2}\leq(|a|+\s)\,b\,.
\end{equation*}
Combined with \eqref{ExpAbsDif} this completes the proof.
\end{proof}

\section{Conclusions and generalizations}\label{S8}

We computed the normalized asymptotic trace norm of random symmetric band circulant matrices and graphs, and estimated the rate of convergence to it (Theorem \ref{AnbThm}), and the probabilities of large deviations (Theorem \ref{LrgDevThm}). We also showed that symmetric band T\"oeplitz matrices and graphs asymptotically have the same normalized trace norms provided their band widths remain small relative to their sizes (Theorem \ref{ToepCirc}). The estimate on the convergence rate is probably optimal in the order of growth although the constants can be improved. One can not expect better than $O\left(\frac1{\sqrt{b}}\right)$ from CLT, and the Bernoullian variables are known to be the worst case \cite[III.11.1]{Shir}. The additional $\ln b$ factor accounts for non-zero means and is due to $\delta$--function like behavior of the Dirichlet kernel for large $b$, it is unlikely to be improvable either. The Talagrand inequality we used produces the optimal order of growth for other matrix ensembles \cite{GZ}, so the estimate in Theorem \ref{LrgDevThm} might be optimal as well. Proving optimality, however, is a different matter that will require new ideas.

The convergence of higher moments is a subtle question. The higher order non--uniform Berry--Esseen inequalities \cite{Daug} can be used to estimate the higher moments as in Lemma \ref{LmSumGauss}, but the lucky cancelation in the last line of \eqref{BerEs} does not occur in general. One then needs a good lower bound for $\sum_{k=1}^b\cos^2(2\pi kt)$ with large $b$, which can be obtained from a lower bound on the minimal value of the Dirichlet kernel. However, there is a bigger issue with the analogs of Lemmas \ref{MeanPart} and \ref{VarPart}. The higher order Lebesgue constants $\L_b^{(p)}:=\big(\frac1{\pi}\int_0^{\pi}|\D_b(t)|^p\,dt\big)^{\frac1p}$ grow as $b^{\frac{p-1}p}$ for $p>1$ \cite{A-S}, so $\frac{\L_b^{(p)}}{\sqrt{b}}$ does not converge to $0$ for $p\geq2$. Therefore, at least the Riemann sum/integral approach that we used will not work, and we have doubts that the spectral moments of order $p\geq2$ converge at all.

Graphs of organic molecules, which served as the original motivation for introducing graph energy, are neither band circulant nor band T\"oeplitz, but benzenoid chains with extremal values of energy do have band block--T\"oeplitz structure 
\cite{RZh1,RZh2}. This means that their adjacency matrices look like band T\"oeplitz matrices with entries replaced by square matrix blocks of fixed size. By the same corner trick from Section \ref{S7} band block--T\"oeplitz matrices can be modified into band block--circulants, and the eigenvalues of the latter can be explicitly expressed via the eigenvalues of the blocks \cite{Tee}. It would be interesting to prove block analogs of Theorems \ref{AnbThm}--\ref{ToepCirc} using matrix--valued versions of CLT and of the Talagrand inequality. Deterministic spectral limits for symmetric block--T\"oeplitz matrices are studied e.g. in \cite{MTi}. It would also be interesting to extend the results of this paper to the Ky Fan norms \cite{NikB}, the incidence energy of graphs \cite{DG}, and skew energy of directed graphs \cite{CLL}.

\end{document}